\documentclass[a4paper,11pt]{article}
\usepackage{amsmath,amsthm,amssymb,color}
\usepackage{amsmath}
\usepackage{amsfonts}

\advance \textheight by \topskip
\advance \textheight by 1pt
\makeatother

\theoremstyle{ams}
\newtheorem{thm}{Theorem}[section]

\newtheorem{lemma}[thm]{Lemma}
\newtheorem{prop}[thm]{Proposition}
\newtheorem{rem}{Remark}

\newcommand{\Ker}{\mathop{\rm Ker}\nolimits}

\newcommand{\im}{\mathop{\rm Im}\nolimits}

\title{The field of meromorphic functions on a sigma divisor of a hyperelliptic curve of genus 3 and applications\footnote{This is the accepted version. 
The final publication is available at link.springer.com}}
\author{T. Ayano\footnote{E-mail: tayano7150@gmail.com}\;\; and V. M. Buchstaber}

\date{}

\begin{document}
\maketitle

\begin{abstract}
The field of meromorphic functions on a sigma divisor of a hyperelliptic curve of genus $3$ is described in terms of the gradient of it's sigma function.
Solutions of corresponding families of polynomial dynamical systems in $\mathbb{C}^4$ with two polynomial integrals are constructed as an application.
These systems were introduced in the work of V. M. Buchstaber and A. V. Mikhailov on the basis of commuting vector fields on the symmetric square of algebraic curves.
\end{abstract}

\section{Introduction}
The Weierstrass elliptic sigma function possesses the following
fundamental properties:

1. Any elliptic function, that is meromorphic
doubly periodic function on $\mathbb{C}$, can be written in the form
$P_1(\wp(u)) + P_2(\wp(u))\wp(u)'$, where $\wp(u)$, $\wp'(u)$ are
the second and third logarithmic derivatives of $\sigma(u)$ and
$P_1(\wp(u)), P_2(\wp(u))$ are rational functions of the Weierstrass function
$\wp(u)$.

2. The functions $\wp(u), \wp(u)'$ and $\wp(u)''$  realize the parametric family of
elliptic curves in the standard Weierstrass model in the form of a cubic surface in $\mathbb{C}^3$.

K. Weierstrass proposed an approach to the construction of Abelian functions on
Jacobians of plane algebraic curves, based on special models of these curves.
F. Klein put forward the program for constructing sigma functions of many variables,
possessing analogues of the properties 1 and 2.
A big contribution to this program was made by H. Baker in his papers, where
for the hyperelliptic curves of genera 2 and 3 he obtained explicit expressions for the
higher logarithmic derivatives of sigma functions of many variables
in the form of polynomials in
logarithmic derivatives of the second and the third order of these functions.
Relatively recently it was shown that these differential polynomials give the
fundamental equations of mathematical physics, including Korteweg--de Vries and
Kadomtzev--Petviashvili equations.
Thus $2g$-periodic meromorphic functions on $\mathbb{C}^g$ which
explicitly give solutions to important equations have been obtained. 
This attracted attention to the theory and applications of sigma functions of many variables.
Variants of this theory with emphasis on different aspects and directions of applications,
can be found in \cite{BEL-97-1}, \cite{O-98}, \cite{O-05} and monography \cite{B2}.
In \cite{BEL-97-1} and in subsequent papers of V.~M.~Buchstaber, D.~V.~Leikin and
V.~Z.~Enolskii for the hyperelliptic curves of genus $g$, it was constructed the theory of sigma functions
$\sigma (w_1, ... ,w_{2g-1})$  as entire functions that are graded homogeneous with respect to variables $w_1, ... ,w_{2g-1}$ and parameters $y_4,...,y_{4g+2}$.
They developed a method to obtain expressions for higher 
logarithmic derivatives of sigma functions in the form of polynomials in
logarithmic derivatives of the second and the third order of these functions.
In terms of $g$-vector functions, which are analogues of the functions $\wp(w),\, \wp(w)',\, \wp(w)''$, 
they obtained the realization of the Jacobian family of hyperelliptic curves of genus $g$ in coordinates 
$w_1, ... ,w_{2g-1}$ in the form of the intersection of cubic surfaces in $\mathbb{C}^{3g}$.
In \cite{BL-08} a general approach to construct Lie algebras of differentiation of hyperelliptic
functions over parameters of the curve was developed.
These Lie algebras were described in an explicit form in \cite{B1} in the case of genus 2 and in \cite{E} in the case of genus 3.

The surface determined by the equation $\sigma(w_1,w_3, ... ,w_{2g-1})=0$ in the Jacobian of a hyperelliptic curve of genus $g$ 
is called the sigma-divisor $(\sigma)$ of this curve.
The set of points of the sigma-divisor $(\sigma)$, in which all the first derivatives $(\sigma_1,\sigma_3,\dots,\sigma_{2g-1}), \; \sigma_k=\frac{\partial}{\partial w_k}\sigma$, $k=2l-1, l=1,\ldots,g$, are zero,
is denoted by $(\sigma)_{\mbox{sing}}$.
The description of the sets $(\sigma)_{\mbox{sing}}$ can be found in \cite{B2} and \cite{F}.
They are subsets of dimension $g-3$ in the Jacobian $\mbox{Jac}(V)$ of the hyperelliptic curve of genus $g>3$,
the empty set for $g=2$, and the point $(0)$ for $g=3$.
Thus, in the case $g=3$ the sigma-divisor $(\sigma)$ has the form $W/\Lambda$, where $W$ is the complex two-dimensional 6-periodic analytic surface in $\mathbb{C}^3$ whose set of singular points coincides with the lattice $\Lambda$ of periods of differentials on this curve.

Let $\mathcal{F}((\sigma))$ be the field of meromorphic functions on the sigma-divisor of the
hyperelliptic curve of genus $3$. The functions $f\in\mathcal{F}((\sigma))$ are considered
as meromorphic functions on $\mathbb{C}^{3}$ whose restrictions to the sigma-divisor ($\sigma$) are 6-periodic.
In this paper a description of the field  $\mathcal{F}((\sigma))$ is  obtained in terms of the gradient of the sigma function.
It essentially uses the operators $L_{k,\ell}$=$\partial_k-\frac{\sigma_k}{\sigma_{\ell}}\partial_{\ell}$
that map the field $\mathcal{F}((\sigma))$ into itself, where $k,\ell=1,3,5$, $k\neq\ell$, and $\partial_k=\frac{\partial}{\partial w_k}$.
As differentiations of the field of meromorphic functions on $\mathbb{C}^3$, they satisfy the commutation relations
$[L_{k_1,\ell},L_{k_2,\ell}]=0$.

As an application we construct solutions of the corresponding families of polynomial dynamical systems on $\mathbb{C}^4$ with two polynomial integrals.
These systems were introduced in \cite{B} on the basis of commuting vector fields on the symmetric squares of plane algebraic curves.
The general approach, developed in \cite{B}, allows us to obtain such systems for more
general models of these curves.

Results and methods of the theory of hyperellitic functions as well as achievements of
the modern singularity theory stimulated the construction and development of applications of
the theory of Abelian functions on Jacobians of algebraic curves of more general models.
A number of results and references to works in this direction can be found in the monograph \cite{B2}.
We also mention the papers \cite{N-10} and \cite{KMP}.

\section{Sigma function}
Set
\[
\Delta=\{(y_4,y_6,y_8,y_{10},y_{12},y_{14})\in\mathbb{C}^6\;|\;\mbox{$Q(X)$\;has a multiple roots}\},
\]
where
\[
Q(X)=X^7+y_4X^5-y_6X^4+y_8X^3-y_{10}X^2+y_{12}X-y_{14},
\]
and $B=\mathbb{C}^6\backslash\Delta$.
Consider the non-singular hyperelliptic curve of genus $3$
\begin{equation}\label{k}
V_{\bf y}=\{(X,Y)\in\mathbb{C}^2\;|\;Y^2=Q(X)\},
\end{equation}
where ${\bf y}=(y_4,y_6,y_8,y_{10},y_{12},y_{14})\in B$.
In this paragraph we recall the definition of the sigma-function for the curve $V_{\bf y}$ (see \cite{B2})
and give the facts about it, which will be used later. For $(X,Y)\in V_{\bf y}$ let
\[
du_1=-\frac{X^2}{2Y}dX,\quad du_3=-\frac{X}{2Y}dX,\quad du_5=-\frac{1}{2Y}dX
\]
be a basis of the vector space of holomorphic 1-forms on $V_{\bf y}$ and $du={}^t(du_1,du_3,du_5)$. Further, let
\[
dr_1=-\frac{X^3}{2Y}dX,\quad dr_3=-\frac{y_4X^2+3X^4}{2Y}dX,\quad dr_5=-\frac{y_8X-2y_6X^2+3y_4X^3+5X^5}{2Y}dX
\]
be a basis of the vector space of meromorphic 1-forms on $V_{\bf y}$ with a pole only at $\infty$.
Let $\{\alpha_i,\beta_i\}_{i=1}^3$ be a canonical basis in the one-dimensional homology group of the curve $V_{\bf y}$.
We define the matrices of periods by
\[
2\omega_1=\left(\int_{\alpha_j}du_{i}\right), \;2\omega_2=\left(\int_{\beta_j}du_{i}\right),\;-2\eta_1=\left(\int_{\alpha_j}dr_i\right), \;-2\eta_2=\left(\int_{\beta_j}dr_i\right).
\]
The normalized matrix of periods has the form $\tau=\omega_1^{-1}\omega_2$.
Let $\delta=\tau\delta'+\delta'',\; \delta',\delta''\in\mathbb{R}^3,$ be the vectors of Riemann's
constants with respect to the choice $(\{\alpha_i,\beta_i\},\infty)$ and $\delta:={}^t({}^t\delta',{}^t\delta'')$.
Then we have $\delta={}^t(\frac{1}{2},\frac{1}{2},\frac{1}{2},\frac{1}{2},0,\frac{1}{2})$.
The sigma-function $\sigma(w)$, $w={}^t(w_1,w_3,w_5)\in \mathbb{C}^3$, is defined by
\[
\sigma(w)=C\exp\left(\frac{1}{2}{}^tw\eta_1\omega_1^{-1}w\right)
\theta[\delta]\bigl((2\omega_1)^{-1}w,\tau\bigr),
\]
where $\theta[\delta](w)$ is the Riemann's theta-function with the characteristics $\delta$, defined by
\[
\theta[\delta](w)=\sum_{n\in\mathbb{Z}^3}\exp(\pi\sqrt{-1}\;{}^t(n+\delta')\tau(n+\delta')+
2\pi\sqrt{-1}\;{}^t(n+\delta')(w+\delta'')),
\]
and $C$ is a constant.
Set $\sigma_i=\partial_i\sigma$ and $\sigma_{i,j}=\partial_i\partial_j\sigma$, where
$\partial_i=\partial/\partial w_i$.
We define the period lattice $\Lambda=\{2\omega_1m_1+2\omega_2m_2\;|\;m_1,m_2\in\mathbb{Z}^3\}$
and set $W=\{w\in\mathbb{C}^3\;|\;\sigma(w)=0\}$.

\begin{prop} {\rm (see \cite{B2} theorem 1.1 and \cite{N-10} p.193)}\label{period}
For $m_1,m_2\in\mathbb{Z}^3$, set $\Omega=2\omega_1m_1+2\omega_2m_2$ and
\[
A=(-1)^{2({}^t\delta'm_1-{}^t\delta''m_2)+{}^tm_1m_2}\exp({}^t(2\eta_1m_1+
2\eta_2m_2)(w+\omega_1m_1+\omega_2m_2)).
\]
Then:

{\rm (i)} $\sigma(w+\Omega)=A\sigma(w)$, where $w\in\mathbb{C}^3$.

{\rm (ii)} $\sigma_i(w+\Omega)=A\sigma_i(w),\;\;i=1,3,5$, where $w\in W$.
\end{prop}

From proposition \ref{period}(i) it follows that $w+\Omega\in W$ for any $w\in W$ and
$\Omega\in\Lambda$. The surface
\[
(\sigma)=\{w\in\mathbb{C}^3/\Lambda\;|\;\sigma(w)=0\}
\]
is called the sigma-divisor $(\sigma)$. We set $\deg w_{2k-1}=-(2k-1)$ and $\deg y_{2i}=2i$, where $k=1,2,3$ and $i=2,\ldots,7$.

\begin{thm} {\rm (see \cite{B2} theorem 7.7, \cite{N-10} theorem 3)} \label{rationallim}
The sigma-function $\sigma(w)=\sigma(w;{\bf y})$ is an entire function on $\mathbb{C}^3$ and is given by the series
\[
\sigma(w)=w_1w_5-w_3^2-\frac{1}{3}w_1^3w_3+\frac{1}{45}w_1^6+
\sum_{i_1+3i_3+5i_5>6}\lambda_{i_1,i_2,i_3}w_1^{i_1}w_3^{i_3}w_5^{i_5},
\]
where the coefficients $\lambda_{i_1,i_2,i_3}\in\mathbb{Q}[y_4,y_6,y_8,y_{10},y_{12},y_{14}]$ are homogeneous polynomials
of degree $i_1+3i_3+5i_5-6$ if $\lambda_{i_1,i_2,i_3}\neq0$.
\end{thm}

Thus, the sigma-function is homogeneous of degree $-6$ with respect to the variables $w$ and the parameters ${\bf y}$.
This fact, which plays an important role in applications, distinguishes the sigma-function from
the Riemann's theta-function, whose series expansion is given by the matrices of periods.

\section{Meromorphic functions on symmetric squares}
In this paragraph we use the description of the field of rational functions on the symmetric
square of the curve $V_{\bf y}$, following the paper \cite{B}. Set ${\bf y}=(y_4,y_6,y_8,y_{10},y_{12},y_{14})\in \mathbb{C}^6$.

Let $\mathcal{F}(V_{\bf y}^2)$ be the field of meromorphic functions on $V_{\bf y}^2$
and $J$ be the ideal in $\mathbb{C}[X_1,Y_1,X_2,Y_2]$ generated by the polynomials $Y_1^2-Q(X_1)$ and $Y_2^2-Q(X_2)$.
We denote the quotient field of an integral domain $R$ by $\langle R\rangle$.
Then we have
\[
\mathcal{F}(V_{\bf y}^2)=\langle\mathbb{C}[X_1,Y_1,X_2,Y_2]/J\rangle.
\]
We introduce the elements $u_2,u_4,u_5,u_7$ of the field of rational functions $\mathbb{C}(X_1,Y_1,X_2,Y_2)$ by
\[
u_2=\frac{X_1+X_2}{2},\quad u_4=\frac{(X_1-X_2)^2}{4},\quad u_5=\left(\frac{Y_1-Y_2}{X_1-X_2}\right),\quad u_7=\frac{Y_1+Y_2}{2}.
\]
Note that the elements $u_2,u_4,u_5,u_7$ are algebraically independent.
\begin{lemma}\label{lem-2}
For $k,\ell\ge0$ we have
\[\frac{1}{2}(X_1^kY_1^{\ell}+X_2^kY_2^{\ell}),\;\frac{1}{2}(X_1^kY_2^{\ell}+X_2^kY_1^{\ell})\in\mathbb{Z}[u_2,u_4,u_5,u_7].\]
\end{lemma}

\begin{proof} Note that $X_1X_2=u_2^2-u_4$ and $Y_1Y_2=u_7^2-u_4u_5^2$. Then:
\begin{multline*}
\sum_{k,\ell=0}^{\infty}(X_1^kY_1^{\ell}+X_2^kY_2^{\ell})s^kt^{\ell}=\frac{1}{1-X_1s}\cdot\frac{1}{1-Y_1t}+
\frac{1}{1-X_2s}\cdot\frac{1}{1-Y_2t}\\
=2\frac{1-u_2s-u_7t+(u_2u_7+u_4u_5)st}{(1-2u_2s+(u_2^2-u_4)s^2)(1-2u_7t+(u_7^2-u_4u_5^2)t^2)}\\
=2(1+u_2s+u_7t+(u_2u_7+u_4u_5)st+(u_4+u_2^2)s^2+(u_4u_5^2+u_7^2)t^2+\ldots).
\end{multline*}
and
\begin{multline*}
\sum_{k,\ell=0}^{\infty}(X_1^kY_2^{\ell}+X_2^kY_1^{\ell})s^kt^{\ell}=\frac{1}{1-X_1s}\cdot\frac{1}{1-Y_2t}
+\frac{1}{1-X_2s}\cdot\frac{1}{1-Y_1t}\\
=2\frac{1-u_2s-u_7t+(u_2u_7-u_4u_5)st}{(1-2u_2s+(u_2^2-u_4)s^2)(1-2u_7t+(u_7^2-u_4u_5^2)t^2)}\\
=2(1+u_2s+u_7t+(u_2u_7-u_4u_5)st+(u_4+u_2^2)s^2+(u_4u_5^2+u_7^2)t^2+\ldots).
\end{multline*}
Comparing the coefficients of $s^kt^{\ell}$  in the generating series, we obtain explicit expressions for the 
functions $X_1^kY_1^{\ell}+X_2^kY_2^{\ell}$ and $X_1^kY_2^{\ell}+X_2^kY_1^{\ell}$ in the form of polynomials in 
$u_2,u_4,u_5,u_7$ for any $k,\ell\ge0$.
\end{proof}

\begin{lemma}\label{16-1}
If a polynomial $f\in\mathbb{Z}[X_1,Y_1,X_2,Y_2]$ satisfies the condition $f(X_1,Y_1,X_2,Y_2)=f(X_2,Y_2,X_1,Y_1)$, 
then $f(X_1,Y_1,X_2,Y_2)\in\mathbb{Z}[u_2,u_4,u_5,u_7]$.
\end{lemma}
\begin{proof}
The polynomial $f$ can be written in the form
\[
\sum_{k_1,k_2,\ell_1,\ell_2} c_{k_1,k_2,\ell_1,\ell_2}(X_1^{k_1}X_2^{k_2}Y_1^{\ell_1}Y_2^{\ell_2}+X_1^{k_2}X_2^{k_1}Y_1^{\ell_2}Y_2^{\ell_1}),
\]
where $c_{k_1,k_2,\ell_1,\ell_2}\in\mathbb{Z}$.
On the other hand, we have:
\begin{itemize}
\item for $k_1\ge k_2$ and $\ell_1\ge\ell_2$
\[
X_1^{k_1}X_2^{k_2}Y_1^{\ell_1}Y_2^{\ell_2}+X_1^{k_2}X_2^{k_1}Y_1^{\ell_2}Y_2^{\ell_1}=
X_1^{k_2}X_2^{k_2}Y_1^{\ell_2}Y_2^{\ell_2}(X_1^{k_1-k_2}Y_1^{\ell_1-\ell_2}+X_2^{k_1-k_2}Y_2^{\ell_1-\ell_2}),
\]
\item for $k_1\ge k_2$ and $\ell_1<\ell_2$
\[
X_1^{k_1}X_2^{k_2}Y_1^{\ell_1}Y_2^{\ell_2}+X_1^{k_2}X_2^{k_1}Y_1^{\ell_2}Y_2^{\ell_1}=
X_1^{k_2}X_2^{k_2}Y_1^{\ell_1}Y_2^{\ell_1}(X_1^{k_1-k_2}Y_2^{\ell_2-\ell_1}+X_2^{k_1-k_2}Y_1^{\ell_2-\ell_1}),
\]
\item for $k_1< k_2$ and $\ell_1\ge\ell_2$
\[
X_1^{k_1}X_2^{k_2}Y_1^{\ell_1}Y_2^{\ell_2}+X_1^{k_2}X_2^{k_1}Y_1^{\ell_2}Y_2^{\ell_1}=
X_1^{k_1}X_2^{k_1}Y_1^{\ell_2}Y_2^{\ell_2}(X_2^{k_2-k_1}Y_1^{\ell_1-\ell_2}+X_1^{k_2-k_1}Y_2^{\ell_1-\ell_2}),
\]
\item for $k_1<k_2$ and $\ell_1<\ell_2$
\[
X_1^{k_1}X_2^{k_2}Y_1^{\ell_1}Y_2^{\ell_2}+X_1^{k_2}X_2^{k_1}Y_1^{\ell_2}Y_2^{\ell_1}=
X_1^{k_1}X_2^{k_1}Y_1^{\ell_1}Y_2^{\ell_1}(X_2^{k_2-k_1}Y_2^{\ell_2-\ell_1}+X_1^{k_2-k_1}Y_1^{\ell_2-\ell_1}).
\]
\end{itemize}
Using the lemma \ref{lem-2}, we obtain the assertion of lemma \ref{16-1}.
\end{proof}

Put
\[
H_{12}=\frac{Y_1^2-Q(X_1)-Y_2^2+Q(X_2)}{X_1-X_2},\quad H_{14}=Y_1^2-Q(X_1)+Y_2^2-Q(X_2).
\]
Then we have
\begin{multline*}
H_{12}(u_2,u_4,u_5,u_7)=2u_5u_7-7u_2^6-35u_2^4u_4-21u_2^2u_4^2-u_4^3-y_4(5u_2^4+10u_2^2u_4+u_4^2)\\
+4y_6(u_2^3+u_2u_4)-y_8(3u_2^2+u_4)+2y_{10}u_2-y_{12},
\end{multline*}
\begin{multline*}
H_{14}(u_2,u_4,u_5,u_7)=-u_7^2-u_4u_5^2+2u_2u_5u_7-6u_2^7-14u_2^5u_4+14u_2^3u_4^2+6u_2u_4^3\\
-4y_4(u_2^5-u_2u_4^2)+y_6(3u_2^4-2u_2^2u_4-u_4^2)-2y_8(u_2^3-u_2u_4)+y_{10}(u_2^2-u_4)-y_{14}.
\end{multline*}
Let $T$ be the ideal in the ring $\mathbb{C}[u_2,u_4,u_5,u_7]$, generated by the polynomials $H_{12}$ and $H_{14}$.
Let $\mbox{Sym}^2(V_{\bf y})$ be the symmetric square of the curve $V_{\bf y}$.
We note that $\mbox{Sym}^2(V_{\bf y}), \;{\bf y}\in B,$ is a complex manifold of dimension $2$.
Let $\mathcal{F}(\mbox{Sym}^2(V_{\bf y}))$ be the field of meromorphic functions on this manifold.
Using the canonical projection $V_{\bf y}^2 \to \mbox{Sym}^2(V_{\bf y})$,
we will consider the field $\mathcal{F}(\mbox{Sym}^2(V_{\bf y}))$ as a subfield of $\mathcal{F}(V_{\bf y}^2)$.
Denote by $\overline{u}_i,\, i=2,4,5,7$, the elements of the field $\mathcal{F}(V_{\bf y}^2)$, which in the field $\mathbb{C}(X_1,Y_1,X_2,Y_2)$ are equal to $u_i$. There is a ring homomorphism
\[
\psi_1\;:\;\mathbb{C}[u_2,u_4,u_5,u_7]\to\mathcal{F}(V_{\bf y}^2),\quad u_i\to\overline{u}_i.
\]

\begin{lemma}\label{16-2}
We have $\Ker\,\psi_1=T$.
\end{lemma}
\begin{proof}
By definition, the polynomials $H_{12}$ and $H_{14}$ lie in $\Ker\,\psi_1$.
Consequently, $T\subset\Ker\,\psi_1$. Put $g\in\Ker\,\psi_1$.
Then there exist a nonnegative integer $a$ and $h\in\mathbb{C}[X_1,Y_1,X_2,Y_2]$ such that
\begin{equation}\label{16-20}
h(X_1,Y_1,X_2,Y_2)=h(X_2,Y_2,X_1,Y_1)\; \text{ and } \; g=h(X_1,Y_1,X_2,Y_2)/u_4^a.
\end{equation}
As $g\in\Ker\,\psi_1$, then there exist $h_1,h_2\in\mathbb{C}[X_1,Y_1,X_2,Y_2]$ such that
\[
h(X_1,Y_1,X_2,Y_2)=(Y_1^2-Q(X_1))h_1(X_1,Y_1,X_2,Y_2)+(Y_2^2-Q(X_2))h_2(X_1,Y_1,X_2,Y_2).
\]
Put $k(X_1,Y_1,X_2,Y_2)=h_1(X_1,Y_1,X_2,Y_2)+h_2(X_2,Y_2,X_1,Y_1)$.
Then it follows from (\ref{16-20}) that
\[
2h(X_1,Y_1,X_2,Y_2)=(Y_1^2-Q(X_1))k(X_1,Y_1,X_2,Y_2)+(Y_2^2-Q(X_2))k(X_2,Y_2,X_1,Y_1).
\]
We have
\[
Y_1^2-Q(X_1)=\frac{1}{2}(H_{14}+(X_1-X_2)H_{12}),\;\;Y_2^2-Q(X_2)=\frac{1}{2}(H_{14}-(X_1-X_2)H_{12}).
\]
Therefore
\begin{multline*}
4h(X_1,Y_1,X_2,Y_2)=(k(X_1,Y_1,X_2,Y_2)+k(X_2,Y_2,X_1,Y_1))H_{14}\\
+(X_1-X_2)(k(X_1,Y_1,X_2,Y_2)-k(X_2,Y_2,X_1,Y_1))H_{12}.
\end{multline*}
Then by lemma \ref{16-1}, there exist  $k_1,k_2\in\mathbb{C}[u_2,u_4,u_5,u_7]$ such that
\[
4h(X_1,Y_1,X_2,Y_2)=k_1(u_2,u_4,u_5,u_7)H_{14}+k_2(u_2,u_4,u_5,u_7)H_{12}.
\]
We set $k_1(u_2,u_4,u_5,u_7)=u_4^b\ell_1(u_2,u_4,u_5,u_7),\; k_2(u_2,u_4,u_5,u_7)=u_4^c\ell_2(u_2,u_4,u_5,u_7)$,
where $\ell_1,\ell_2\in\mathbb{C}[X_1,Y_1,X_2,Y_2]$ do not divide by $u_4$.
Without loss of generality, we can assume that $b\le c$. Set $a>b$. Then $4u_4^{a-b}g=H_{14}\ell_1+H_{12}u_4^{c-b}\ell_2$.
Since the right-hand side of this equality is not divisible by $u_4$, we arrive at a contradiction.
Hence $a\le b$, and we get
\[
4g=H_{14}u_4^{b-a}\ell_1+H_{12}u_4^{c-a}\ell_2\in T.
\]
\end{proof}

There is an embedding of fields $\widetilde{\psi}_1\colon \langle\mathbb{C}[u_2,u_4,u_5,u_7]/T\rangle \to \mathcal{F}(V_{\bf y}^2)$.
\begin{thm}\label{t-4}
We have $\mathcal{F}(\mbox{\rm Sym}^2(V_{\bf y}))=\im\widetilde{\psi}_1$.
\end{thm}
\begin{proof}
The field $\mathcal{F}(V_{\bf y}^2)$ is an extension of the field $\mathbb{C}(X_1,X_2)$ by means of two quadratic elements
$Y_1$ and $Y_2$. So any element $g$ of $\mathcal{F}(V_{\bf y}^2)$ can be uniquely written in the form
\begin{equation}\label{a}
g=h_1(X_1,X_2)Y_1Y_2+h_2(X_1,X_2)Y_1+h_3(X_1,X_2)Y_2+h_4(X_1,X_2),
\end{equation}
where $h_i(X_1,X_2)\in \mathbb{C}(X_1,X_2)$ for $1\le i\le 4$.
We have
\begin{equation}\label{b}
\mathcal{F}(\mbox{Sym}^2(V_{\bf y}))=\{g\in\mathcal{F}(V_{\bf y}^2)\;|\;
h_i(X_2,X_1)=h_i(X_1,X_2),i=1,4,\;h_2(X_2,X_1)=h_3(X_1,X_2)\}.
\end{equation}
Thus $\overline{u}_2,\overline{u}_4,\overline{u}_5,\overline{u}_7\in\mathcal{F}(\mbox{Sym}^2(V_{\bf y}))$, and so
$\im\widetilde{\psi}_1\subset \mathcal{F}(\mbox{Sym}^2(V_{\bf y}))$.

We now prove that the reverse inclusion holds.
Let $f(X_1,X_2)\in\mathbb{C}(X_1,X_2), \; f(X_1,X_2)\neq0$. Then
\[
f(X_1,X_2)=\frac{f_1(X_1,X_2)}{f_2(X_1,X_2)},\;f_i(X_1,X_2)\in\mathbb{C}[X_1,X_2],
\]
where $f_1(X_1,X_2)$ and $f_2(X_1,X_2)$ are relatively prime polynomials.

\begin{lemma}\label{lem1}
If $f(X_2,X_1)=f(X_1,X_2)$, then we have
\[
f_1(X_2,X_1)=f_1(X_1,X_2),\;\;f_2(X_2,X_1)=f_2(X_1,X_2).
\]
\end{lemma}

\begin{proof}
Since $f(X_2,X_1)=f(X_1,X_2)$, it follows that
\[
f_1(X_2,X_1)f_2(X_1,X_2)=f_1(X_1,X_2)f_2(X_2,X_1).
\]
The ring $\mathbb{C}[X_1,X_2]$ is a unique factorization domain. Since
$f_1(X_1,X_2)$ and $f_2(X_1,X_2)$ are relatively prime,
there exists a polynomial $k(X_1,X_2)\in\mathbb{C}[X_1,X_2]$ such that
\begin{equation}\label{prime}
f_1(X_2,X_1)=f_1(X_1,X_2)k(X_1,X_2),\;\;f_2(X_2,X_1)=f_2(X_1,X_2)k(X_1,X_2).
\end{equation}
Using now that the polynomials $f_1(X_2,X_1)$ and $f_2(X_2,X_1)$ are also
relatively prime, we have $c=k(X_1,X_2)\in\mathbb{C}^*$.
From the condition that the polynomials $f_1(X_1,X_2)$ and $f_2(X_1,X_2)$
are relatively prime, we have that at least one of them is not divisible by $X_1-X_2$.
Without loss of generality, we can assume that $f_1(X_1,X_2)$ is not divisible by $X_1-X_2$.
Then $f_1(X_1,X_1)\neq0$ and according to (\ref{prime}) we get $f_1(X_1,X_1)=cf_1(X_1,X_1)$,
i.e. $c=1$.
\end{proof}

We write an arbitrary element $g\in\mathcal{F}(\mbox{Sym}^2(V_{\bf y}))$ in the form (\ref{a}).
From (\ref{b}) we find that $h_1(X_1,X_2)$ and $h_4(X_1,X_2)$ are symmetric functions.
If $h_1(X_1,X_2), h_4(X_1,X_2)\neq0$, then according to lemma \ref{lem1} their numerators and
denominators are the symmetric polynomials.
The map (see \cite{B} lemma 11)
\[
\xi\colon \mathbb{C}^2\times\mathbb{C}^2 \to\mathbb{C}^5; \quad \xi((X_1,Y_1),(X_2,Y_2))=(v_2,v_4,v_7,v_9,v_{14}),
\]
where $v_2=X_1+X_2,\; v_4=(X_1-X_2)^2,\; v_7=Y_1+Y_2,\; v_9=(X_1-X_2)(Y_1-Y_2),\; v_{14}=(Y_1-Y_2)^2$, allows us to identify the manifold $(\mathbb{C}^2\times\mathbb{C}^2)/S_2$ with the hypersurface in $\mathbb{C}^5$, given by the equation $v_4v_{14}-v_9^2=0$. Here $S_2$ is
the group of permutations on a set of two elements. Using that
$v_2=2u_2,\; v_4=4u_4,\; v_7=2u_7,\; v_9=4u_4u_5,\; v_{14}=4u_4u_5^2$, we get $h_1(X_1,X_2),\;h_4(X_1,X_2)\in\mathbb{C}(u_2,u_4,u_5,u_7)$.

Set $h_2(X_1,X_2)=0$ in the expansion (\ref{a}) for the element $g$.
Then from (\ref{b}) we have $h_3(X_1,X_2)=0$. Therefore, in this case we obtain $g\in\im\widetilde{\psi}_1$.
Now we consider the case $h_2(X_1,X_2)\neq0$ and
write the functions $h_2(X_1,X_2)$ and $h_3(X_1,X_2)$ in the form of irreducible fractions
\[
h_2(X_1,X_2)=\frac{h_{21}(X_1,X_2)}{h_{22}(X_1,X_2)},\;\;h_3(X_1,X_2)=\frac{h_{31}(X_1,X_2)}{h_{32}(X_1,X_2)},
\]
where $h_{ij}(X_1,X_2)\in\mathbb{C}[X_1,X_2]$.

\begin{lemma}\label{lem3}
There exists a constant $c\in\mathbb{C}^*$ such that
\[
h_{31}(X_1,X_2)=c\cdot h_{21}(X_2,X_1),\quad h_{32}(X_1,X_2)=c\cdot h_{22}(X_2,X_1).
\]
\end{lemma}

\begin{proof}
Since $h_2(X_2,X_1)=h_3(X_1,X_2)$, we have $h_{21}(X_2,X_1)h_{32}(X_1,X_2)=h_{22}(X_2,X_1)h_{31}(X_1,X_2)$.
Using again that the ring $\mathbb{C}[X_1,X_2]$ is a unique factorization domain, we obtain
that there exists a polynomial $\ell(X_1,X_2)\in\mathbb{C}[X_1,X_2]$ such that
\[
h_{31}(X_1,X_2)=h_{21}(X_2,X_1)\ell(X_1,X_2),\quad h_{32}(X_1,X_2)=h_{22}(X_2,X_1)\ell(X_1,X_2).
\]
Since $h_{31}(X_1,X_2)$ and $h_{32}(X_1,X_2)$ are relatively prime,
we have $c=\ell(X_1,X_2)\in\mathbb{C}^*$.
\end{proof}

Now let us return to the proof of the theorem \ref{t-4}. We have
\begin{equation}\label{c}
h_2(X_1,X_2)Y_1+h_3(X_1,X_2)Y_2=\frac{h_{21}(X_1,X_2)h_{32}(X_1,X_2)Y_1+h_{22}(X_1,X_2)h_{31}
(X_1,X_2)Y_2}{h_{22}(X_1,X_2)h_{32}(X_1,X_2)}.
\end{equation}
According to lemma \ref{lem3}, there is an equality $h_{22}(X_1,X_2)h_{32}(X_1,X_2)=
c\cdot h_{22}(X_1,X_2)h_{22}(X_2,X_1)$. Hence the polynomial $h_{22}(X_1,X_2)h_{32}(X_1,X_2)$
is symmetric and the denominator of the fraction in the formula (\ref{c}) is a polynomial in $u_2$ and $u_4$.
According to lemma \ref{lem3}, there is an equality
\begin{multline*}
h_{21}(X_1,X_2)h_{32}(X_1,X_2)Y_1+h_{22}(X_1,X_2)h_{31}(X_1,X_2)Y_2 \\
=c(h_{21}(X_1,X_2)h_{22}(X_2,X_1)Y_1+h_{22}(X_1,X_2)h_{21}(X_2,X_1)Y_2).
\end{multline*}
Therefore, the numerator in the formula (\ref{c}) is a symmetric polynomial.
Then according to lemma \ref{16-1}, it belongs to the ring $\mathbb{C}[u_2,u_4,u_5,u_7].$
Thus we get $h_2(X_1,X_2)Y_1+h_3(X_1,X_2)Y_2\in\mathbb{C}(u_2,u_4,u_5,u_7)$,
that is $g\in\im\widetilde{\psi}_1$.
\end{proof}

\section{Meromorphic functions on sigma-divisor}

Let us consider the Abel--Jacobi map
\[
I\colon\mbox{Sym}^2(V_{\bf y})\to{\rm Jac}(V_{\bf y})=\mathbb{C}^3/\Lambda,\qquad
(P_1,P_2)\to\int_{\infty}^{P_1}du+\int_{\infty}^{P_2}du.
\]
Set $w^{[2]}=I((P_1,P_2))$ where $P_i=(X_i,Y_i)\in V_{\bf y}$, $i=1,2$.
We introduce the following meromorphic functions on $\mathbb{C}^3$
\[
f_1=\frac{\sigma_{1,1}}{\sigma_1},\;f_2=\frac{\sigma_3}{\sigma_1},\;
f_3=\frac{\sigma_{1,3}}{\sigma_1},\;f_4=\frac{\sigma_5}{\sigma_1},\;
f_5=\frac{\sigma_{3,3}}{\sigma_1},\;g_5=\frac{\sigma_{1,5}}{\sigma_1},\;
f_7=\frac{\sigma_{3,5}}{\sigma_1}.
\]

Note that the indices of these functions are chosen equal to the
grading of the functions.
The set $W=\{w\in \mathbb{C}^3\;|\;\sigma(w)=0\}$ is a complex
2-dimensional surface in $\mathbb{C}^3$.
The gradient of the sigma-function $\nabla\sigma=(\sigma_1,\sigma_3,\sigma_5)$ defines
the complex curves $W_i=\{w\in W\;|\;\sigma_i(w)=0\}$ on this surface.
It is known that $W_1\cap W_3\cap W_5=\Lambda$ (see \cite{B2}, \cite{F}).

From proposition \ref{period} (ii) it follows that $w+\Omega\in W_i$ for any
$w\in W_i$ and $\Omega\in\Lambda$.
Thus, the complex curves $W_i/\Lambda\subset{\rm Jac}(V_{\bf y}),\;i=1,3,5$, are defined.
For any $j=1,3,5$, the two meromorphic on $\mathbb{C}^3$ functions $\sigma_i/\sigma_j$, $i\neq j$, are
holomorphic on $W\backslash W_j$ and invariant under the shift of any point $w\in W$
by periods $\Omega\in\Lambda$.
We note that only the functions $f_2=\frac{\sigma_3}{\sigma_1}$ and $f_4=\frac{\sigma_5}{\sigma_1}$
among the seven functions listed above give single-valued functions on the sigma-divisor $W/\Lambda$.

\begin{prop}\label{prop3}
We have the following formulas:
\[
\overline{u}_2=-\frac{1}{2}f_2(w^{[2]}),\;\;\overline{u}_4=(\frac{1}{4}f_2^2-f_4)(w^{[2]}),\;\;\overline{u}_5=\frac{1}{2}(f_1f_2^2+f_5-2f_2f_3)(w^{[2]}),
\]
\[
\overline{u}_7=\frac{1}{4}(2f_2^2f_3-2f_3f_4-f_1f_2^3+2f_1f_2f_4-f_2f_5+2f_7-2f_2g_5)(w^{[2]}).
\]
\end{prop}

The following result is used to prove this assertion:
\begin{lemma}\label{lem4} 
The function $\sigma_1(w^{[2]})$ as a function of the points $P_1$ and $P_2$
is not identically equal to zero. 
We have
\begin{align}
X_1+X_2 &=-f_2(w^{[2]}),\label{equ} \\
X_1X_2 &=f_4(w^{[2]}). \label{equ1}
\end{align}
\end{lemma}
\begin{proof}
Set
\[
\wp_{i,j}(w)=-\frac{\partial^2}{\partial w_i\partial w_j}\log \sigma(w).
\]
From \cite{B2} theorem 2.4 (see also \cite{A} theorem 7.3) it follows that the Abel--Jacobi map
\[
I_3 \colon \mbox{Sym}^3(V_{\bf y})\to{\rm Jac}(V_{\bf y})
\]
induces a homomorphism $I_3^*$ such that
\[
I_3^*\wp_{1,1}(w)=X_1+X_2+X_3,\;\;I_3^*\wp_{1,3}(w)=-X_1X_2-X_2X_3-X_3X_1,\;\;
I_3^*\wp_{1,5}(w)=X_1X_2X_3.
\]
Since
\[
\wp_{1,1}=\frac{\sigma_1^2-\sigma\sigma_{1,1}}{\sigma^2},\;\;\wp_{1,3}=
\frac{\sigma_1\sigma_3-\sigma\sigma_{1,3}}{\sigma^2},\;\;\wp_{1,5}=
\frac{\sigma_1\sigma_5-\sigma\sigma_{1,5}}{\sigma^2},
\]
we have
\[
\frac{\sigma_1\sigma_3-\sigma\sigma_{1,3}}{\sigma_1^2-\sigma\sigma_{1,1}}=
-\frac{X_1X_2+X_2X_3+X_3X_1}{X_1+X_2+X_3},\;\;\frac{\sigma_1\sigma_5-\sigma\sigma_{1,5}}{\sigma_1^2-\sigma\sigma_{1,1}}=
\frac{X_1X_2X_3}{X_1+X_2+X_3}.\label{inversion}
\]
Calculating the limit as $X_3\to\infty$, we obtain the assertion.
\end{proof}

\begin{proof}[Proof of proposition \ref{prop3}]
The formulas for $\overline{u}_2$ and $\overline{u}_4$ follow from (\ref{equ}) and (\ref{equ1}).
By differentiating (\ref{equ}) with respect to $P_i$, $i=1,2$, we get
\[
-dX_i=f_{2,1}(w^{[2]})\left(-\frac{X_i^2}{2Y_i}dX_i\right)+
f_{2,3}(w^{[2]})\left(-\frac{X_i}{2Y_i}dX_i\right)+f_{2,5}(w^{[2]})\left(-\frac{1}{2Y_i}dX_i\right),
\]
where $f_{2,i}=\partial_if_2$.
Thus, we have the following relation between the meromorphic functions of the points $P_1$ and $P_2$
\begin{equation}
2Y_i=f_{2,1}(w^{[2]})X_i^2+f_{2,3}(w^{[2]})X_i+f_{2,5}(w^{[2]}).\label{equ2}
\end{equation}
Direct calculations yield the formulas
\[
f_{2,1}=f_3-f_1f_2,\quad f_{2,3}=f_5-f_2f_3,\quad f_{2,5}=f_7-f_2g_5.
\]
From lemma \ref{lem4} and (\ref{equ2}), we have
\[
2\overline{u}_5=f_{2,1}(w^{[2]})(X_1+X_2)+f_{2,3}(w^{[2]})=(f_1f_2^2+f_5-2f_2f_3)(w^{[2]}),
\]
\[
4\overline{u}_7=(2f_2^2f_3-2f_3f_4-f_1f_2^3+2f_1f_2f_4-f_2f_5+2f_7-2f_2g_5)(w^{[2]}).
\]
\end{proof}

Let $\mathcal{F}$ be the field of all meromorphic functions on $\mathbb{C}^3$ and
$\mathcal{F}[(\sigma)]$ be the set of meromorphic functions $f\in\mathcal{F}$ satisfying the following two conditions:
\begin{itemize}
\item for any point $w\in W$ there exist its open neighborhood $U\subset\mathbb{C}^3$
and two holomorphic functions $g$ and $h$ on $U$ such that the function
$h$ is not identically equal to zero on $U\cap W$ and $f=g/h$ on $U$;
\item $f(w+\Omega)=f(w)$ for any $w\in W$ and $\Omega\in\Lambda$.
\end{itemize}
Note that $\mathcal{F}[(\sigma)]$ is a subring in $\mathcal{F}$, but it is not a field in general.

We introduce the following meromorphic functions on $\mathbb{C}^3$: 
\[
F_2=-\frac{1}{2}f_2,\quad F_4=\frac{1}{4}f_2^2-f_4,\quad F_5=\frac{1}{2}(f_1f_2^2+f_5-2f_2f_3),
\]
\[
F_7=\frac{1}{4}(2f_2^2f_3-2f_3f_4-f_1f_2^3+2f_1f_2f_4-f_2f_5+2f_7-2f_2g_5).
\]
According to proposition \ref{prop3},  we have $F_2,F_4,F_5,F_7\in\mathcal{F}[(\sigma)]$.
The Abel-Jacobi map $I$ induces a ring homomorphism
\[
I^* \colon \mathcal{F}[(\sigma)]\to\mathcal{F}(\mbox{Sym}^2(V_{\bf y})),\quad f\to f\circ I.
\]
From the definition of $F_i$, it follows that $I^*(F_i)=\overline{u}_i$ for $i=2,4,5,7$.

According to theorem \ref{t-4}, the homomorphism $I^*$ is an epimorphism.
Let $J^*$ be the set of meromorphic functions $f\in\mathcal{F}[(\sigma)]$ such that $f$ is identically equal to zero on $W$.
Thus we have Ker $I^*=J^*$.
Set $\mathcal{F}((\sigma))=\mathcal{F}[(\sigma)]/J^*$.
It is obvious that $\mathcal{F}((\sigma))$ is a field and, by construction, there is an isomorphism of fields
\begin{equation}
\overline{I^*} \;: \;\mathcal{F}((\sigma)) \to \mathcal{F}(\mbox{Sym}^2(V_{\bf y})).\label{cong}
\end{equation}
Let $\overline{F}_i$ be the equivalence class of $F_i$ in $\mathcal{F}((\sigma))$. 
Consider the ring homomorphism
\[
\psi_2 \colon \mathbb{C}[u_2,u_4,u_5,u_7]\to\mathcal{F}((\sigma)),\quad u_i\to \overline{F}_i.
\]
\begin{thm}\label{mainthe}
We have the equality $\Ker\,\psi_2=T$ and, consequently, the isomorphism
\[
\widetilde{\psi}_2\colon \langle\mathbb{C}[u_2,u_4,u_5,u_7]/T\rangle \longrightarrow \mathcal{F}((\sigma)).
\]
\end{thm}
\begin{proof} The result follows directly from
lemma \ref{16-2}, theorem \ref{t-4}, and formula (\ref{cong}).
\end{proof}

\section{Derivations of $\mathcal{F}(\mbox{Sym}^2(V_{\bf y}))$}
The following two {\it commuting} derivations of the field $\mathbb{C}(X_1,Y_1,X_2,Y_2)$ of rational functions were used in \cite{B}:
\[
\mathcal{L}_3^*=\frac{1}{X_1-X_2}(\mathcal{D}_2-\mathcal{D}_1)\footnote{In \cite{B},  
the operator $-\mathcal{L}_3^*$ (in our notation) was used.},\quad \mathcal{L}_5^*=\frac{1}{X_1-X_2}(X_2\mathcal{D}_1-X_1\mathcal{D}_2),
\]
where
\[
\mathcal{D}_k=2Y_k\partial_{X_k}+Q'(X_k)\partial_{Y_k},\;\;k=1,2.
\]
Since
\[
\mathcal{L}_3^*(Y_k^2-Q(X_k))=0,\;\;\;\mathcal{L}_5^*(Y_k^2-Q(X_k))=0,\;\;\;k=1,2,
\]
the operators $\mathcal{L}_3^*,\mathcal{L}_5^*$ can be regarded as derivations of the field $\mathcal{F}(V_{\bf y}^2)=\langle\mathbb{C}[X_1,Y_1,X_2,Y_2]/J\rangle$.

\begin{thm} {\rm (\cite{B} lemmas 16 and 17)}\label{system1}
In the space $\mathbb{C}^4$ with coordinates $u_2,u_4,u_5,u_7$, we have the following families
of dynamical systems with constant parameters $y_4,y_6,y_8,y_{10}$:
\[
\mathcal{L}_3^*u_2=-u_5,\;\;\mathcal{L}_3^*u_4=-2u_7,
\]
\[
\mathcal{L}_3^*u_5=-35u_2^4-42u_2^2u_4-3u_4^2-2y_4(5u_2^2+u_4)+4y_6u_2-y_8,
\]
\[
\mathcal{L}_3^*u_7=-7(3u_2^5+10u_2^3u_4+3u_2u_4^2)-10y_4(u_2^3+u_2u_4)+2y_6(3u_2^2+u_4)-3y_8u_2+y_{10},
\]
\[
\mathcal{L}_5^*u_2=u_2u_5-u_7,\;\;\mathcal{L}_5^*u_4=2(u_2u_7-u_4u_5),
\]
\[
\mathcal{L}_5^*u_5=u_5^2+14u_2^5-28u_2^3u_4-18u_2u_4^2-8y_4u_2u_4+2y_6(u_2^2+u_4)-2y_8u_2+y_{10},
\]
\[
\mathcal{L}_5^*u_7=-u_5u_7+21u_2^6+35u_2^4u_4-21u_2^2u_4^2-3u_4^3+2y_4(5u_2^4-u_4^2)-
\]
\[
-2y_6(3u_2^3-u_2u_4)+y_8(3u_2^2-u_4)-y_{10}u_2.
\]
\end{thm}

Using theorem \ref{system1}, we can identify the operators $\mathcal{L}_3^*$ and $\mathcal{L}_5^*$
with the derivation operators of the field $\mathcal{F}(\mbox{Sym}^2(V_{\bf y}))$.

\section{Derivation of $\mathcal{F}((\sigma))$}
In this paragraph we will introduce the {\it commuting} derivations $L_3^*$ and $L_5^*$ of the field
$\mathcal{F}((\sigma))$ such that $\mathcal{L}_3^*\circ \overline{I^*}=\overline{I^*}\circ L_3^*$ and
$\mathcal{L}_5^*\circ \overline{I^*}=\overline{I^*}\circ L_5^*$.
We will obtain an expression for these operators in terms of the derivation operators
$\partial_1,\partial_3,\partial_5$ of functions on $\mathbb{C}^3$ with coordinates
$w_1,w_3,w_5$.

Let $\omega$ and $\eta$ be meromorphic 1-forms on $V_{\bf y}$.
Then there exists a unique meromorphic function $f$ on $V_{\bf y}$ such that
$\omega=f\cdot\eta$. We will denote such $f$ by $\omega/\eta$.
Consider the following operators on the field $\mathcal{F}(V_{\bf y}^2)$:
\[
\widetilde{\mathcal{L}_3^*}=\frac{1}{X_1-X_2}\{-2Y_1(d_{P_1}/dX_1)+2Y_2(d_{P_2}/dX_2)\},
\]
\[
\widetilde{\mathcal{L}_5^*}=\frac{1}{X_1-X_2}\{2X_2Y_1(d_{P_1}/dX_1)-2X_1Y_2(d_{P_2}/dX_2)\}.
\]
Let us describe in more detail the action of these operators. Let $g(P_1,P_2)\in \mathcal{F}(V_{\bf y}^2)$.
Fix $P_2$ then $d_{P_1}(g)$ is the total derivative of $g$ as a meromorphic function of $P_1$
and $dX_1$ is the total derivative of the meromorphic function $X_1$ on $V_{\bf y}$.
As noted above, $d_{P_1}(g)/dX_1$ can be regarded as a meromorphic function on $V_{\bf y}$.
Treating the point $P_2$ as a variable, we obtain the function $d_{P_1}(g)/dX_1$ as a meromorphic function on $V_{\bf y}^2$.
Similarly we obtain a meromorphic function $d_{P_2}(g)/dX_2$ on $V_{\bf y}^2$.
Therefore the operator $\widetilde{\mathcal{L}_3^*}$ transforms the meromorphic function $g(P_1,P_2)$ into
the meromorphic function
\[
\frac{1}{X_1-X_2}(-2Y_1(d_{P_1}(g)/dX_1)+2Y_2(d_{P_2}(g)/dX_2)).
\]
The action of $\widetilde{\mathcal{L}_5^*}$ on $\mathcal{F}(V_{\bf y}^2)$ is described similarly.
Using the Leibniz rule for the total derivative, we get that $\widetilde{\mathcal{L}_3^*}$ and $\widetilde{\mathcal{L}_5^*}$ are
the derivations of the field $\mathcal{F}(V_{\bf y}^2)$.

\begin{prop}\label{gh}
We have the equalities $\widetilde{\mathcal{L}_3^*}=\mathcal{L}_3^*$ and
$\widetilde{\mathcal{L}_5^*}=\mathcal{L}_5^*$ as the derivations of $\mathcal{F}(V_{\bf y}^2)$.
\end{prop}

\begin{proof}
From the definition of the operators $\mathcal{L}_3^*$ and $\mathcal{L}_5^*$ we obtain
\[
\mathcal{L}_3^*X_1=\frac{-2Y_1}{X_1-X_2},\;\;\mathcal{L}_3^*Y_1=\frac{-Q'(X_1)}{X_1-X_2},
\;\;\mathcal{L}_3^*X_2=\frac{2Y_2}{X_1-X_2},\;\;\mathcal{L}_3^*Y_2=\frac{Q'(X_2)}{X_1-X_2},
\]
\[
\mathcal{L}_5^*X_1=\frac{2X_2Y_1}{X_1-X_2},\;\;\mathcal{L}_5^*Y_1=\frac{X_2Q'(X_1)}{X_1-X_2},
\;\;\mathcal{L}_5^*X_2=\frac{-2X_1Y_2}{X_1-X_2},\;\;\mathcal{L}_5^*Y_2=\frac{-X_1Q'(X_2)}{X_1-X_2}.
\]
From the relations $Y_i^2=Q(X_i)$, $i=1,2$, it follows that
\[
\frac{dY_i}{dX_i}=\frac{Q'(X_i)}{2Y_i}.
\]
Thus we get
\[
\widetilde{\mathcal{L}_3^*}X_1=\frac{-2Y_1}{X_1-X_2},\;\;\widetilde{\mathcal{L}_3^*}Y_1=
\frac{-Q'(X_1)}{X_1-X_2},\;\;\widetilde{\mathcal{L}_3^*}X_2=\frac{2Y_2}{X_1-X_2},\;\;
\widetilde{\mathcal{L}_3^*}Y_2=\frac{Q'(X_2)}{X_1-X_2},
\]
\[
\widetilde{\mathcal{L}_5^*}X_1=\frac{2X_2Y_1}{X_1-X_2},\;\;\widetilde{\mathcal{L}_5^*}Y_1=
\frac{X_2Q'(X_1)}{X_1-X_2},\;\;\widetilde{\mathcal{L}_5^*}X_2=\frac{-2X_1Y_2}{X_1-X_2},\;\;
\widetilde{\mathcal{L}_5^*}Y_2=\frac{-X_1Q'(X_2)}{X_1-X_2}.
\]
Therefore $\mathcal{L}_3^*=\widetilde{\mathcal{L}_3^*}$ and $\mathcal{L}_5^*=\widetilde{\mathcal{L}_5^*}$.
\end{proof}

Let us consider a general case.
For any holomorphic function $R(w)\in \mathcal{F}$ on $\mathbb{C}^3$ such that $R_i(w)\not\equiv 0$, where
$R_i=\partial_iR$,\, $i=1,3,5$, we set
\[
Z_R=\{w\in\mathbb{C}^3\;|\;R(w)=0\}.\]
Let $\mathcal{F}[Z_R]$ be the set of meromorphic functions $f\in\mathcal{F}$ such that, for any point $w\in Z_R$, 
there exists an open neighborhood $U\subset\mathbb{C}^3$ of $w$ and two holomorphic functions $g$ and $h$ on $U$ such that the function $h$ 
is not identically equal to zero on $U\cap Z_R$ and $f=g/h$ on $U$. 
Set 
\[\mathcal{F}(Z_R)=\mathcal{F}[Z_R]/\sim,\]
where $f\sim g$ if and only if there exists a function $h\in\mathcal{F}[Z_R]$ such that $f-g=h\cdot R$.

We introduce the following 6 operators
\[
L_{k,\ell}=\partial_k-\frac{R_k}{R_{\ell}}\partial_{\ell},\;\;\;k,\ell=1,3,5,\;\;k\neq\ell.
\]

\begin{lemma}\label{commutator}
{\rm 1.} We have the commutation relations
\[
[L_{1,3},L_{5,3}]=[L_{1,5},L_{3,5}]=[L_{3,1},L_{5,1}]=0
\]
in the Lie algebra of the derivations of $\mathcal{F}$.

{\rm 2.} The operators $L_{k,\ell}$ define the derivations of the ring $\mathcal{F}(Z_R)$.
\end{lemma}

The proof of the lemma follows from the definition of the operators $L_{k,\ell}$
and direct calculation of the commutators in lemma \ref{commutator}(1).

Let $w^{(0)}=(w_1^{(0)},w_3^{(0)},w_5^{(0)})\in Z_R$ be a point such that $R_1(w^{(0)})\neq0$.
Consider a sufficiently small open neighborhood $U\subset\mathbb{C}^2$ around $(w_3^{(0)},w_5^{(0)})$.
Then, according to the implicit function theorem, there exists a unique holomorphic on $U$ function
$\varphi(w_3,w_5)$ such that
\begin{itemize}
\item $\varphi(w_3^{(0)},w_5^{(0)})=w_1^{(0)}$,
\item $(\varphi(w_3,w_5),w_3,w_5)\in Z_R$ for any point $(w_3,w_5)\in U$.
\end{itemize}

\begin{lemma}\label{deri}
For any function $F\in\mathcal{F}$, we have the formulas
\[
\frac{\partial}{\partial w_3}F(\varphi(w_3,w_5),w_3,w_5)=L_{3,1}(F),\quad \frac{\partial}{\partial w_5}F(\varphi(w_3,w_5),w_3,w_5)=L_{5,1}(F).
\]
\end{lemma}

\begin{proof}
According to the definition of the function $\varphi$ we have
\[
\frac{\partial \varphi}{\partial w_3}=-\frac{R_3}{R_1}(\varphi(w_3,w_5),w_3,w_5),\quad \frac{\partial \varphi}{\partial w_5}=-\frac{R_5}{R_1}(\varphi(w_3,w_5),w_3,w_5).
\]
Therefore
\[
\frac{\partial}{\partial w_3}F(\varphi(w_3,w_5),w_3,w_5)=-\frac{R_3}{R_1}(\partial_1F)+\partial_3F=L_{3,1}(F),
\]
\[
\frac{\partial}{\partial w_5}F(\varphi(w_3,w_5),w_3,w_5)=-\frac{R_5}{R_1}(\partial_1F)+\partial_5F=L_{5,1}(F).
\]
\end{proof}

\begin{rem}
Another proof of lemma \ref{deri} follows from the relation $[\partial_3,\partial_5]=0$ and lemma \ref{commutator}.
\end{rem}

Now let us return to the case $R(w)=\sigma(w)$, $L_3^*=L_{3,1}$ and $L_5^*=L_{5,1}$, i.e.,
\[
L_3^*=-f_2\partial_1+\partial_3,\quad L_5^*=-f_4\partial_1+\partial_5.
\]

\begin{lemma}\label{expression}

{\rm (i)} Set $h\in\mathcal{F}[(\sigma)]$ then $L_3^*(h),L_5^*(h)\in\mathcal{F}[(\sigma)]$.

\vspace{2ex}

{\rm (ii)} Set $h\in J^*$ then $L_3^*(h),L_5^*(h)\in J^*$.
Therefore we can regard the operators $L_3^*$ and $L_5^*$ as derivations of the field $\mathcal{F}((\sigma))=\mathcal{F}[(\sigma)]/J^*$.

{\rm (iii)} We have the following formulas $\mathcal{L}_3^*\circ \overline{I^*}=\overline{I^*}\circ L_3^*$ and $\mathcal{L}_5^*\circ \overline{I^*}=\overline{I^*}\circ L_5^*$, where $\overline{I^*} \colon \mathcal{F}((\sigma)) \to \mathcal{F}(\mbox{Sym}^2(V_{\bf y}))$ {\rm(see paragraph 4)}.
\end{lemma}

\begin{proof}
Set $h\in\mathcal{F}[(\sigma)]$. According to proposition \ref{gh}, we have
\begin{multline*}
\mathcal{L}_3^*\circ I^*(h)=\mathcal{L}_3^*(h(w^{[2]}))
=\frac{1}{X_1-X_2}(-2Y_1(d_{P_1}(h(w^{[2]}))/dX_1)+2Y_2(d_{P_2}(h(w^{[2]}))/dX_2)) \\
=\frac{1}{X_1-X_2}\left(-2Y_1\left(-\frac{X_1^2}{2Y_1}h_1(w^{[2]})-\frac{X_1}{2Y_1}h_3(w^{[2]})-
\frac{1}{2Y_1}h_5(w^{[2]})\right)\right)\\
+\frac{1}{X_1-X_2}\left(2Y_2\left(-\frac{X_2^2}{2Y_2}h_1(w^{[2]})-\frac{X_2}{2Y_2}h_3(w^{[2]})-
\frac{1}{2Y_2}h_5(w^{[2]})\right)\right)\\
=(X_1+X_2)h_1(w^{[2]})+h_3(w^{[2]}),
\end{multline*}
where $h_i=\partial_ih$. Since the operator $\mathcal{L}_3^*$ transforms
$\mathcal{F}(\mbox{Sym}^2(V_{\bf y}))$ into itself, we get
\[
(X_1+X_2)h_1(w^{[2]})+h_3(w^{[2]})\in\mathcal{F}(\mbox{Sym}^2(V_{\bf y})).
\]
Since $-(X_1+X_2)=f_2(w^{[2]})$, we have $L_3^*(h)=-f_2h_1+h_3\in\mathcal{F}[(\sigma)]$.
Therefore we obtain a proof of (i) for the operator $L_3^*$. 
Set $h\in J^*$, then $(X_1+X_2)h_1(w^{[2]})+h_3(w^{[2]})=0$ as an element of $\mathcal{F}(\mbox{Sym}^2(V_{\bf y}))$.
This means that $L_3^*(h)=-f_2h_1+h_3\in J^*$. Therefore we obtain a proof of (ii) for the operator $L_3^*$.
We have
\[
I^*\circ L_3^*(h)=I^*(-f_2h_1+h_3)=(X_1+X_2)h_1(w^{[2]})+h_3(w^{[2]}).
\]
Therefore, $\mathcal{L}_3^*\circ I^*(h)=I^*\circ L_3^*(h)$.
This proves (iii) for the operator $L_3^*$.

The lemma's assertions for the operator $L_5^*$ are proved similarly.
\end{proof}

\section{Dynamical systems on $\mathbb{C}^4$ integrable by the functions on the sigma-divisor}

For any constant vector $(y_4,y_6,y_8,y_{10})\in\mathbb{C}^4$, we consider the following dynamical systems
on $\mathbb{C}^4$ with coordinates $(G_2,G_4,G_5,G_7)$

(I)
\[
\frac{\partial G_2}{\partial t}=-G_5,\quad \frac{\partial G_4}{\partial t}=-2G_7,
\]
\[
\frac{\partial G_5}{\partial t}=-35G_2^4-42G_2^2G_4-3G_4^2-2y_4(5G_2^2+G_4)+4y_6G_2-y_8,
\]
\[
\frac{\partial G_7}{\partial t}=-7(3G_2^5+10G_2^3G_4+3G_2G_4^2)-1
0y_4(G_2^3+G_2G_4)+2y_6(3G_2^2+G_4)-3y_8G_2+y_{10},
\]

\vspace{1ex}

(II)
\[
\frac{\partial G_2}{\partial \tau}=G_2G_5-G_7,\quad \frac{\partial G_4}{\partial \tau}=2(G_2G_7-G_4G_5),
\]
\[
\frac{\partial G_5}{\partial \tau}=G_5^2+14G_2^5-28G_2^3G_4-18G_2G_4^2-8y_4G_2G_4+2y_6(G_2^2+G_4)-2y_8G_2+y_{10},
\]
\begin{multline*}
\frac{\partial G_7}{\partial \tau}=-G_5G_7+21G_2^6+35G_2^4G_4-21G_2^2G_4^2-3G_4^3+2y_4(5G_2^4-G_4^2)- \\
-2y_6(3G_2^3-G_2G_4)+y_8(3G_2^2-G_4)-y_{10}G_2.
\end{multline*}
Consider the following polynomials on $\mathbb{C}^4$:
\begin{multline*}
I_{12}(G_2,G_4,G_5,G_7)=2G_5G_7-7G_2^6-35G_2^4G_4-21G_2^2G_4^2-G_4^3-y_4(5G_2^4+10G_2^2G_4+G_4^2) \\
+4y_6(G_2^3+G_2G_4)-y_8(3G_2^2+G_4)+2y_{10}G_2,
\end{multline*}
\begin{multline*}
I_{14}(G_2,G_4,G_5,G_7)=-G_7^2-G_4G_5^2+2G_2G_5G_7-6G_2^7-14G_2^5G_4+14G_2^3G_4^2+6G_2G_4^3 \\
-4y_4(G_2^5-G_2G_4^2)+y_6(3G_2^4-2G_2^2G_4-G_4^2)-2y_8(G_2^3-G_2G_4)+y_{10}(G_2^2-G_4).
\end{multline*}

\begin{thm}
For any constant vector $(y_4,y_6,y_8,y_{10})\in\mathbb{C}^4$, any solution $(G_2,G_4,G_5,G_7)$ of the dynamical systems {\rm(I}) and {\rm(II)} has the polynomial integrals $I_i,\;i=12,14$, i.e., 
\[
\frac{\partial}{\partial t}I_i(G_2(t,\tau),G_4(t,\tau),G_5(t,\tau),G_7(t,\tau))=0,
\]
\[
\frac{\partial}{\partial\tau}I_i(G_2(t,\tau),G_4(t,\tau),G_5(t,\tau),G_7(t,\tau))=0.
\]
\end{thm}

\begin{proof}
The direct calculation of the partial derivatives of the polynomials
\[\widehat{I}_i(t,\tau)=I_i(G_2(t,\tau),G_4(t,\tau),G_5(t,\tau),G_7(t,\tau))\]
in solutions of systems (I) and (II) shows that they are identically equal to zero.
\end{proof}

Consider a constant vector $(y_4,y_6,y_8,y_{10},y_{12},y_{14})\in B$.
Take a point $w^{(0)}=(w_1^{(0)},w_3^{(0)},w_5^{(0)})\in W$ such that $\sigma_1(w^{(0)})\neq0$.
Then in a sufficiently small open neighborhood $U\subset\mathbb{C}^2$ around $(w_3^{(0)},w_5^{(0)})$,
there exists a uniquely determined holomorphic function $\varphi(w_3,w_5)$ on $U$ such that
\[
\varphi(w_3^{(0)},w_5^{(0)})=w_1^{(0)} \; \text{ and }\; (\varphi(w_3,w_5),w_3,w_5)\in W\;
\text{ for any point }\; (w_3,w_5)\in U.
\]

Let $F_i(w_1,w_3,w_5),\;i=2,4,5,7$, be the functions on the sigma-divisor introduced in paragraph 4.
Set $G_i(t,\tau)=F_i(\varphi(t,\tau),t,\tau)$.

\begin{thm}\label{dynamical}
The functions
\[
G_i(t,\tau)=F_i(\varphi(t,\tau),t,\tau),\;\;i=2,4,5,7,
\]
satisfy the dynamical systems {\rm (I)} and {\rm (II)}.
\end{thm}

The proof of this theorem follows immediately from
theorem \ref{system1}, lemma \ref{deri} and lemma \ref{expression}.

\section{Rational limit}
Let the constant vector ${\bf y}\in\mathbb{C}^6$ tend to zero.
Then according to theorem \ref{rationallim}, the sigma-function $\sigma(w_1,w_3,w_5)$ passes
in the Shur--Weierstrass polynomial (see \cite{BEL-99-R})
\[
\sigma=w_1w_5-w_3^2-\frac{1}{3}w_1^3w_3+\frac{1}{45}w_1^6.
\]
As a result, we obtain
\[
\sigma_1=w_5-w_1^2w_3+\frac{2}{15}w_1^5,\quad \sigma_3=-2w_3-\frac{1}{3}w_1^3,\quad \sigma_5=w_1,
\]
\[
\sigma_{11}=-2w_1w_3+\frac{2}{3}w_1^4,\quad \sigma_{13}=-w_1^2,\quad \sigma_{15}=1,
\quad \sigma_{33}=-2,\quad \sigma_{35}=0.
\]

Let $p\in\mathbb{C}$ be such that $p^6=-45p$ and $w^{(0)}=(p,0,1)$. Then
$w^{(0)}\in W$,\, $\sigma_1(w^{(0)})=1$ for $p=0$ and $\sigma_1(w^{(0)})=-5$ for $p\neq0$.
Take a sufficiently small open neighborhood $U\subset\mathbb{C}^2$ around $(0,1)$.
Then according to the implicit function theorem
there exists a unique holomorphic function $\varphi(w_3,w_5)$ on $U$ such that
\[
\varphi(0,1)=p\; \text{ and }\; (\varphi(w_3,w_5),w_3,w_5)\in W\; \text{ for any point } (w_3,w_5)\in U.
\]
Set $w_3=t$ and $w_5-1=\tau$. Then from the relation
\begin{equation}
\varphi(w_3,w_5)w_5-w_3^2-\frac{1}{3}\varphi(w_3,w_5)^3w_3+\frac{1}{45}\varphi(w_3,w_5)^6=0,\label{ratrel}
\end{equation}
we obtain the recurrence relation
\[
\varphi(t,1+\tau)=t^2-\tau\varphi(t,1+\tau)+\frac{1}{3}t\varphi(t,1+\tau)^3-\frac{1}{45}\varphi(t,1+\tau)^6.
\]
Therefore
\[
\varphi(t,1+\tau)=\sum_{k,\ell\ge0}\gamma_{k,\ell}t^k\tau^{\ell},\quad
\gamma_{k,\ell}\in\mathbb{Z}\left[p,\frac{1}{3},\frac{1}{5}\right].
\]
From the definitions of $F_2,F_4,F_5,F_7$ (see paragraph 4) taking into account the relation (\ref{ratrel}), we have
\[
F_2(\varphi(w_3,w_5),w_3,w_5)=(5/K_2)(\varphi^3+6w_3),
\]
\[F_4(\varphi(w_3,w_5),w_3,w_5)=(15/K_4)(-\varphi^3w_3+15\varphi w_5-15w_3^2),
\]
\[
F_5(\varphi(w_3,w_5),w_3,w_5)=(1/K_5)(8\varphi^5w_5+3\varphi^4w_3^2-15\varphi^2w_3w_5-45\varphi w_3^3-15w_5^2),
\]
\[
F_7(\varphi(w_3,w_5),w_3,w_5)=(10125/(2K_7))(15\varphi^5w_3w_5-50\varphi^4w_3^3-25\varphi^3w_5^2+129\varphi^2w_3^2w_5-111\varphi w_3^4),
\]
where
\[
K_2=2(2\varphi^5-15\varphi^2w_3+15w_5),\quad K_4=4(-8\varphi^5w_5+27\varphi^4w_3^2-30\varphi^2w_3w_5+15w_5^2),
\]
\[
K_5=3(14\varphi^5w_5^2-111\varphi^4w_3^2w_5+189\varphi^3w_3^4+
165\varphi^2w_3w_5^2-585\varphi w_3^3w_5+405w_3^5+5w_5^3),
\]
\begin{multline*}
K_7=729\varphi^5w_3^5-208\varphi^5w_5^3+3042\varphi^4w_3^2w_5^2-11583\varphi^3w_3^4w_5+2187\varphi^2w_3^6- \\
-4380\varphi^2w_3w_5^3+28620\varphi w_3^3w_5^2-24300w_3^5w_5+15w_5^4.
\end{multline*}

\begin{thm}\label{rationaldy}
The set of functions $(G_2,G_4,G_5,G_7)$, where
\[
G_i(t,\tau)=F_i(\varphi(t,1+\tau),t,1+\tau),\;\;i=2,4,5,7,
\]
is a solution of the dynamical systems {\rm (I)} and {\rm (II)} with $y_4=y_6=y_8=y_{10}=0$.
\end{thm}

This theorem follows from theorem \ref{dynamical},
because the parametric families of solutions obtained in the latter can be transformed into solutions of systems (I) and (II) with $y_4=y_6=y_8=y_{10}=0$ 
by passing to the limit.

\vspace{1ex}

\noindent{\bf Example 1.} Set $p=0$. Put $\psi_1(t)=\varphi(t,1)$.
The function $\psi_1(t)$ has the following expansion in a neighborhood of the point $t=0$
\[
\psi_1(t)=t^2+\frac{1}{3}t^7+\frac{44}{45}t^{12}+\cdots.
\]
According to theorem \ref{rationaldy}, the set of functions $(G_2,G_4,G_5,G_7)$, where
\[
G_i(t)=F_i(\psi_1(t),t,1),\;\;i=2,4,5,7,
\]
is a solution of the dynamical system (I) with $y_4=y_6=y_8=y_{10}=0$.

\vspace{1ex}

\noindent{\bf Example 2.} Set $p^5=-45$. Put $\psi_2(\tau)=\varphi(0,1+\tau)$.
Then $\psi_2(\tau)=p(1+\tau)^{1/5}$ and we have
\[
F_2(\psi_2(\tau),0,1+\tau)=-\frac{1}{30}p^3(1+\tau)^{-2/5},\quad F_4(\psi_2(\tau),0,1+\tau)=\frac{3}{20}p(1+\tau)^{-4/5},
\]
\[
F_5(\psi_2(\tau),0,1+\tau)=\frac{1}{5}(1+\tau)^{-1},\quad F_7(\psi_2(\tau),0,1+\tau)=-\frac{1}{50}p^3(1+\tau)^{-7/5}.
\]
According to theorem \ref{rationaldy}, the set of functions $(G_2,G_4,G_5,G_7)$, where
\[
G_i(\tau)=F_i(\psi_2(\tau),0,1+\tau),\;\;i=2,4,5,7,
\]
is a solution of the dynamical system (II) with $y_4=y_6=y_8=y_{10}=0$.

\vspace{1ex}

\noindent{\bf Example 3.}  Let $q\in\mathbb{C}$ be such that $q^6=15q^3+45$. 
Then $(q(1+t)^{1/3},1+t,0)\in W$ and $\sigma_1(q(1+t)^{1/3},1+t,0)=q^2(\frac{2}{15}q^3-1)(1+t)^{5/3}$.
We have
\[
F_2(q(1+t)^{1/3},1+t,0)=\frac{q}{6}(1+t)^{-2/3},\quad F_4(q(1+t)^{1/3},1+t,0)=
-\frac{q^2(q^3+15)}{108(q^3+3)}(1+t)^{-4/3},
\]
\[
F_5(q(1+t)^{1/3},1+t,0)=\frac{q}{9(8q^3+21)}(1+t)^{-5/3},
\]
\[
F_7(q(1+t)^{1/3},1+t,0)=-\frac{q^2(287q^3+750)}{1458(7q^3+18)}(1+t)^{-7/3}.
\]
According to theorem \ref{rationaldy}, the set of functions $(G_2,G_4,G_5,G_7)$, where
\[
G_i(t)=F_i(q(1+t)^{1/3},1+t,0),\;\;i=2,4,5,7,
\]
is a solution of the dynamical system (I) with $y_4=y_6=y_8=y_{10}=0$.

\vspace{2ex}

{\bf Acknowledgment.} The authors are grateful to V. Z. Enolskii for a useful discussion of the results of this paper and references.

\end{document}